\documentclass{amsart}
\usepackage{amssymb}
\usepackage{a4wide}
\usepackage[all]{xypic}
\usepackage{graphicx}

\DeclareMathOperator{\Aut}{Aut}
\DeclareMathOperator{\GL}{GL}

\newcommand{\Z}{{\mathbb Z}}

\newtheorem{theorem}{Theorem}[section]
\newtheorem{proposition}[theorem]{Proposition}
\newtheorem{lemma}[theorem]{Lemma}
\newtheorem{definition}[theorem]{Definition}
\newtheorem{corollary}[theorem]{Corollary}
\newtheorem*{remark}{Remark}

\begin{document}

\title{The $R_\infty$--property for nilpotent quotients of Baumslag--Solitar groups.}
\author[K.\ Dekimpe]{Karel Dekimpe}\thanks{K.~Dekimpe is supported by long term structural funding -- Methusalem grant of the Flemish Government.}
\author[D.\ Gon\c{c}alves]{Daciberg Gon\c{c}alves}\thanks{D. ~L. ~Gon\c calves is partially supported by Projecto Tem\'atico  Topologia Alg\'ebrica, Geom\'etrica e Diferencial FAPESP no. 2016/24707-4.}
\address{KU Leuven Campus Kulak Kortrijk\\
E. Sabbelaan 53\\
8500 Kortrijk\\
Belgium}
\email{karel.dekimpe@kuleuven.be}
\address{Departemento de Matem\'atica--IME--USP\\ 
Universidade de S\~{a}o Paulo\\ 
S\~{a}o Paulo\\
Brasil}
\email{dlgoncal@ime.usp.br}

\maketitle

\begin{abstract} A group $G$ has the $R_{\infty}$--property if the number $R(\varphi)$ of twisted conjugacy classes is infinite for any automorphism $\varphi$ of $G$. For such a group $G$, the $R_{\infty}$--nilpotency index is the least integer $c$ such that 
$G/\gamma_{c+1}(G)$ still has the $R_\infty$--property. In this paper,  we determine the  $R_{\infty}$--nilpotency 
degree 
of all Baumslag--Solitar groups.  
\end{abstract}
\section{Introduction}

Any endomorphism $\varphi$ of a group $G$ determines an equivalence relation on $G$ by setting $x\sim y \Leftrightarrow 
\exists z \in G:\; x = z y \varphi(z)^{-1}$. The equivalence classes of this relation are called Reidemeister classes or twisted conjugacy classes and their number is denoted by $R(\varphi)$. We are most interested in this number when 
$\varphi$ is an automorphism. 

\medskip

For information on the  development, historical aspects and the relation of this concept with other topics in mathematics such as fixed point theory, we refer the reader to the introduction of \cite{DG2}  and its references. An important concept in this context is that of groups having the $R_\infty$--property.

\begin{definition}  A group $G$ is said to have  the $R_{\infty}$--property if for every automorphism $\varphi: G \to G$  the number 
 $R(\varphi)$ is  infinite.
 \end{definition}
 
  A central problem is to decide which groups have 
 the $R_{\infty}$--property.  The study of this problem
  has  been a quite active  research topic in recent years.
  Several families of groups have been studied by many authors. A non-exhaustive list of references is \cite{DG, DG1, DG2, Fe, FG,FN, FT, GW, GW1,JLL,STW}.

\medskip

Of  particular interest for this paper is the fact that in \cite{FG} it was proved that the  Baumslag-Solitar groups $BS(m,n)$ have the  $R_{\infty}$--property except for 
$m=n=1$ (or $m=n=-1$ which is the same group). 
Recently in \cite{DG2}, motivated by the results of \cite{DG},    new examples of
groups which have the    $R_{\infty}$--property  were obtained 
by looking at quotients of a  group which has the $R_{\infty}$--property by the terms of the lower central series as well the derived central series.  
 So it is natural to  ask the same question for the groups $BS(m,n)$.

Related to this approach, we introduced in \cite{DG2} the following notion:
   
\begin{definition}Let $G$ be a group. The $R_{\infty}$--nilpotency 
degree of a group  $G$  is the least integer $c$ such that $G/\gamma_{c+1}(G)$
has the $R_{\infty}$--property. In case this integer does not exist then we say that $G$ has $R_{\infty}$--nilpotency degree infinite.
\end{definition}

   In this work we determine the  $R_{\infty}$--nilpotency degree for all the Baumslag--Solitar groups
$BS(m,n)$. 
 The main results of this work are:
 
 \medskip
 
 {\bf Theorem 4.5}
Let $m,n$ be integers with $0<m\leq |n|$ and $\gcd(m,n)=1$. Let $p$ denote the largest integer such that 
$2^p|2 m +2$. Then, the $R_\infty$--nilpotency degree $r$ of $BS(m,n)$ is given by 
\begin{itemize}
\item In case $n<0$ and $n\neq -1$, then $r=2$.
\item In case $n=-1$ (so $m=1$) then $r=\infty$.
\item In case $n=m$ (so $n=m=1$) then $r=\infty$.
\item In case $n-m=1$, then $r=\infty$.
\item In case $n-m=2$, then $r=p+2$.
\item In case $n-m\geq 3$, then $r=2$.
\end{itemize}

and   

\medskip

{\bf Theorem 5.4}
Let $0<m \leq |n|$ with $m\neq n$ and take $d=\gcd(m,n)$. 
Let $p$ denote the largest integer such that 
$2^p|2 \frac{m}{d} +2$. Then, the $R_\infty$--nilpotency degree $r$ of $BS(m,n)$ is given by 
\begin{itemize}
\item In case $n<0$ and $n\neq -m$, then $r=2$.
\item In case $n=-m$ then $r=\infty$.
\item In case $n=m$  then $r=\infty$.
\item In case $n-m=d$, then $r=\infty$.
\item In case $n-m=2d$, then $2\leq r\leq p+2$.
\item In case $n-m\geq 3d$, then $r=2$.
\end{itemize}
 
 \medskip
 
 At this point we would also like to mention 
one interesting family  of groups which extends naturally the class of  Baumslag--Solitar groups, namely the family of $GBS$ groups, the Generalized Baumslag--Solitar groups. In \cite{Le} the following strong result about the Reidemeister number 
of a homomorphims of such groups is proved:
 
\medskip
 
{\bf Proposition \cite[Proposition 2.7]{Le}}  Let $\alpha : G \to G$ be an endomorphism of a non-elementary 
$GBS$ group. If one of the following conditions holds, then $R(\alpha)$ is infinite:\\
(1) $\alpha$ is surjective.\\
(2) $\alpha$ is injective and $G$ is not unimodular.\\
(3) $G = BS(m, n)$ with $|m|\ne   |n|$, and the image of $\alpha$ is not cyclic.

\medskip

 As a generalization of the results of this paper, it would be natural to study the $R_{\infty}$--nilpotency degree for the $GBS$ groups.

\medskip 

This work is divided into  3 sections besides the introduction. In section~2  we provide some preliminary results about the description of the terms of the lower central series and the corresponding quotients of $BS(m,n)$, notoriously when the 
integers $(m,n)$ are coprime.      In section~3 we construct certain specific nilpotent groups in a format which is
convenient for our study. Then we identify these groups which the ones that we want to study, namely the quotients 
$BS(m,n)/\gamma_{c+1}(BS(m,n))$.   In section 4, we then show the main result for the $BS(m,n)$ groups where $m$ and $n$ are coprime. Finally, in section 5  we provide a proof for the remaining cases.

\section{Baumslag--Solitar groups}
 Let \[ BS(m,n)=\langle a, b \;|\; a^{-1} b^m a = b^n \rangle \] for $m,n$ integers. It suffices to consider 
$1 \leq m \leq |n|$. We will use the notation   $[x, y]=x^{-1}y^{-1}xy$. 
\begin{lemma}\label{order} Consider a Baumslag--Solitar group $BS(m,n)$.\\
For all positive integers $k$ we have that 
\[ b^{(m-n)^k}\in \gamma_{k+1}(BS(m,n)).\]
\end{lemma}
\begin{proof}
As $a^{-1} b^{-m} a = b^{-n}$, we have that $b^{m-n}=[a,b^m]\in \gamma_2(BS(m,n))$, which proves the lemma for $k=1$.

Now, we assume  that $k\geq 1$ and that  $b^{(m-n)^k}\in \gamma_{k+1}(BS(m,n))$. Then we find:
\begin{eqnarray*}
\lefteqn{b^{-m(m-n)^k}\in \gamma_{k+1}(BS(m,n))}\\
& \Rightarrow & a^{-1}  b^{-m(m-n)^k} a  b^{m(m-n)^k} \in \gamma_{k+2}(BS(m,n))\\
& \Rightarrow & \left( a^{-1} b^m a\right)^{-(m-n)^k} b^{m(m-n)^k} \in \gamma_{k+2}(BS(m,n))\\
& \Rightarrow & b^{  -n (m-n)^k } b^{m(m-n)^k}= b^{(m-n)^{k+1}} \in \gamma_{k+2}(BS(m,n))
\end{eqnarray*}
which proves the lemma, by induction.
\end{proof}

Since we will be dealing with nilpotent quotients of the Baumslag--Solitar groups, 
we introduce the following notation
\[ BS_c(m,n)=  \frac{BS(m,n)}{\gamma_{c+1}(BS(m,n))}.\]
For a nilpotent group $N$, we use $\tau N$ to indicate its torsion subgroup.

\begin{lemma} \label{hirsch}Let $m\neq n$.
For all positive integers $c$, the nilpotent group $BS_c(m,n)$ has Hirsch length 1 and 
if we denote by $\bar{b}$ the natural projection of $b$ in $BS_c(m,n)$, we have that 
\[ \tau BS_c(m,n)= \langle \bar{b}, \gamma_2( BS_c(m,n) )\rangle.\]
\end{lemma}
\begin{proof}
We first consider the case $c=1$.
Note that 
\[ BS_1(m,n)=\langle \bar a,\bar b\,|\, [\bar a,\bar b]=1,\; \bar b^{m-n}=1\rangle \cong \Z \oplus \Z_{|m-n|}.\]
So $\tau BS_1(m,n)= \langle \bar b \rangle$.

Now, let $c>1$.  From the case $c=1$, it follows that $\tau BS_c(m,n) \subseteq \langle \bar b,\; \gamma_2(BS_c(m,n)) \rangle$ and hence it 
suffices to show that $\gamma_2 (BS_c(m,n))$ is a torsion group. To obtain this result,
we prove by induction on $i\geq 2$ that $\gamma_i(BS_c(m,n))/\gamma_{i+1}(BS_c(m,n))=\gamma_i(BS(m,n))/\gamma_{i+1}(BS(m,n))$
is finite.

The group $\gamma_2(BS(m,n))/\gamma_3(BS(m,n))$ is generated by $[a,b]\gamma_3(BS(m,n))$. By the previous lemma,
we know that $b^{m-n}\in \gamma_2(BS(m,n))$, using this we find:
\[ [a,b]^{m-n} \gamma_3(BS(m,n)) = [ a, b^{m-n}] \gamma_3(BS(m,n)) = 1 \gamma_3(BS(m,n)) \]
and so $ [a,b]\gamma_3(BS(m,n))$ is of finite order ($\leq |m-n|$) in $\gamma_2(BS(m,n))/\gamma_3(BS(m,n))$.

\medskip

Now, assume that $\gamma_i(BS(m,n))/\gamma_{i+1}(BS(m,n))$ is finite. The group $\gamma_{i+1}(BS(m,n))/\gamma_{i+2}(BS(m,n))$ is generated by all elements of the form $[x,y]\gamma_{i+2}(BS(m,n))$ where $x\in BS(m,n)$ and $y\in\gamma_i(BS(m,n))$. By our assumption, there is a $k>0$ so that $y^k\in \gamma_{i+1}(BS(m,n))$. As before   then follows that 
$[x,y]^k\gamma_{i+2}(BS(m,n)) = [x,y^k]\gamma_{i+2}(BS(m,n)) = 1 \gamma_{i+2}(BS(m,n))$,  from which we deduce that 
$\gamma_{i+1}(BS(m,n))/\gamma_{i+2}(BS(m,n))$ is finite.

\medskip

The fact that $BS_c(m,n)$ has Hirsch length 1, follows from the fact that $BS_c(m,n)/\gamma_2(BS_c(m,n))\cong BS_1(m,n)$ has Hirsch length 1 
and $\gamma_2(BS_c(m,n))$ has Hirsch length 0.

\end{proof}

In this paper the situation where $\gcd(m,n)=1$ will play a rather crucial role. For these groups, the structure of
$BS_c(m,n)$   is easier to understand than in the general case. E.g., we have the following lemma.
\begin{lemma} Suppose that $\gcd(m,n)=1$ and $m\neq n$. For any $c>1$ and $k>1$, we have that 
\[ \gamma_k (BS_c(m,n))=\langle \bar b^{(m-n)^{k-1}} \rangle.\] 
Again $\bar b$ denotes the projection of $b$ in $BS_c(m,n)$.
\end{lemma}
\begin{proof} For sake of simplicity, we will write $\Gamma_i$     instead 
of $\gamma_i(BS_c(m,n))$ in the rest of this proof.
We will prove by induction on $k\geq 2$ that $\bar  b^{(m-n)^{k-1}} \Gamma_{k+1}$ generates 
$ \Gamma_{k}/ \Gamma_{k+1}$.

For $k=2$, we have that $[\bar a,\bar b]\Gamma_3$ generates $\Gamma_2/\Gamma_3$ and 
from Lemma ~\ref{order} we know that $[\bar a, \bar b]^{m-n} \in \Gamma_3$, hence, the 
order of  $[\bar a,\bar b]\Gamma_3$ in $\Gamma_2/\Gamma_3$ is a divisor of $m-n$. As $\gcd(m,n)=1$, also 
$\gcd(m,m-n)=1$ and therefore also $[\bar a,\bar b]^m\Gamma_3$ is a generator of $\Gamma_2/\Gamma_3$.
Now,
\[ [\bar a, \bar b]^m \Gamma_3 = [ \bar a ,\bar b^m]\Gamma_3 = \bar b^{m-n}\Gamma_3,\]
 from which we find that $\bar b^{m-n}\Gamma_3$ generates $\Gamma_2/\Gamma_3$.

\medskip

Now, we assume that $k> 2$ and that $\Gamma_{k-1}/\Gamma_k$ is generated by $\bar b^{(m-n)^{k-2}}\Gamma_k$. The 
next quotient $\Gamma_k/\Gamma_{k+1}$ is then generated by $[\bar a, \bar b^{(m-n)^{k-2}}]\Gamma_{k+1}$.
By Lemma~\ref{order} again, we have that 
\[ [\bar a, \bar b^{(m-n)^{k-2}}]^{m-n}\Gamma_{k+1}=  [\bar a, \bar b^{(m-n)^{k-1}}] \Gamma_{k+1}= 1 \Gamma_{k+1}\]
and so the order of the generator $[\bar a, \bar b^{(m-n)^{k-2}}]\Gamma_{k+1}$ divides $m-n$. As before, it follows that 
also $[\bar a, \bar b^{(m-n)^{k-2}}]^m\Gamma_{k+1}$ generates $\Gamma_{k}/\Gamma_{k+1}$. 
In $BS(m,n)$ we have that $[a,b^{km}]= b^{k(m-n)}$, which we now use to obtain that 
\[ [\bar a, \bar b^{(m-n)^{k-2}}]^m\Gamma_{k+1} = [\bar a, \bar b^{ m (m-n)^{k-2} }]\Gamma_{k+1}=\bar{b}^{(m-n)^{k-1}}\Gamma_{k+1},\]
which finishes the proof. 
\end{proof}

\begin{corollary}\label{metabelian}
Suppose that $\gcd(m,n)=1$ and $m\neq n$, then for all $c\geq 1$ we have that $\tau BS_c(m,n)= \langle \bar b \rangle$
\end{corollary}

\section{Some nilpotent quotients of Baumslag--Solitar groups.}
For the rest of this section we assume that $m\neq n$. 
For any positive integer $c$, we will construct a nilpotent group $G_c(m,n)$ of class $\leq c$ which can be seen 
as a quotient of $BS(m,n)$. To construct this group, we fix $m$, $n$ and $c$   and consider the 
morphism 
$\varphi: \Z^c \rightarrow \Z^c$ which is represented by the matrix:
\begin{equation}\label{matrix-phi}
\left( \begin{array}{cccccc}
n-m & 0 & 0 & \cdots & 0 & 0 \\
-m & n-m & 0 & \cdots & 0 & 0 \\
0 & -m & n-m & \cdots & 0 & 0\\
\vdots & \vdots & \vdots & \ddots & \vdots & \vdots  \\
0 & 0 & 0 & \cdots & n-m& 0\\
0 & 0 & 0 & \cdots & -m & n-m
\end{array}\right).
\end{equation}
Here we use the convention that elements of $\Z^c$ are written as columns, so also in the matrix above, the image of the $i$-th standard generator of $\Z^c$ is given by the $i$-th column of that matrix. We now consider the 
abelian group 
\[ A_c(m,n)= \frac{\Z^c}{{\rm Im}\, \varphi}.\]
So $A_c(m,n)$ is a finite group of order $|n-m|^c$. 

We consider also the morphism $\psi:\Z^c\rightarrow \Z^c$, which is represented by 
\[ M= \left( \begin{array}{cccccc}
1 & 0 & 0 & \cdots & 0 & 0 \\
1 & 1& 0 & \cdots & 0 & 0 \\
0 & 1 & 1 & \cdots & 0 & 0\\
\vdots & \vdots & \vdots & \ddots & \vdots & \vdots  \\
0 & 0 & 0 & \cdots & 1& 0\\
0 & 0 & 0 & \cdots & 1 & 1
\end{array}\right).\]
We have that $\varphi \psi = \psi \varphi$ and therefore $\psi$ induces an automorphism of   $A_c(m,n)$, which we will also denote by the symbol $\psi$. 

Now, we are ready to define the group $G_c(m,n)$, which is given as a semi-direct product
\[  G_c(m,n) = A_c(m,n) \rtimes \langle t \rangle, \]
where $\langle t \rangle$ is the infinite cyclic group and where the semi-direct product structure is given by the requirement that 
\[ \forall a \in A_c(m,n): t^{-1} a t = \psi(a). \]

For any $z\in \Z^c$, we use $\overline{z}=z + {\rm Im}\,\varphi$ to denote its natural projection in $A_c(m,n)$. We let 
$e_1,\,e_2,\, \ldots,\, e_c$ denote the standard generators of $\Z^c$, so $e_i$ is the column vector having a 1 on the $i$-th spot and $0$'s on all other positions. Obviously, we have that 
$\overline{e_1}, \overline{e_2}, \ldots, \overline{e_c}$ generate $A_c(m,n)$. 
For  sake of simplicity, sometimes we will write $G$ instead of $G_c(m,n)$.
\begin{remark}
It is easy to see that  from the fact that 
\[ \overline{e_2} = t^{-1} \overline{e_1} t \overline{e_1}^{-1}, \;
\overline{e_3} = t^{-1} \overline{e_2} t \overline{e_2}^{-1},\; \cdots\]
follows 
\[ \gamma_2(G) \subseteq \langle \overline{e_2}, \, \overline{e_3},\, \ldots ,\, \overline{e_c} \rangle \]
\[ \gamma_3(G) \subseteq \langle \overline{e_3}, \, \overline{e_4},\, \ldots ,\, \overline{e_c} \rangle \]
\[ \vdots \]
\[ \gamma_c(G) \subseteq \langle  \overline{e_c} \rangle \]
\[ \gamma_{c+1}(G) =1\]
hence $G_c(m,n)$ is nilpotent of class $\leq c$.
\end{remark}

\begin{lemma}\label{relation}
There is a surjective morphism of groups 
\[ f: BS(n,m) \rightarrow G_c(m,n) \]
which is determined by 
\[ f(a) = t \mbox{ and } f(b)= \overline{e_1}.\]
\end{lemma} 
\begin{proof}
In order for $f$ to be a morphism, we need to check that $f$ preserves the defining relation of $BS(n,m)$, that is 
the relation 
\[
t^{-1} \overline{e_1}^m t = \overline{e_1}^n 
\]
should hold. This follows from the following computation:
\begin{eqnarray*}
t^{-1} \overline{e_1}^m t  & = & \psi (\overline{e_1}^m) \\
& = & \overline{e_1}^m \overline{e_2}^m\\
& = & \overline{e_1}^n \left( \overline{e_1}^{m-n} \overline{e_2}^m\right) \\
& = &  \overline{e_1}^n \overline{\varphi( -e_1)}\\
& = & \overline{e_1}^n .
\end{eqnarray*}
To prove that $f$ is a surjective map, it is enough to show that $\overline{e_1}$ and $t$ generate $G_c(m,n)$.
This follows from the fact that 
\[ \overline{e_2} = t^{-1} \overline{e_1} t \overline{e_1}^{-1}, \;
\overline{e_3} = t^{-1} \overline{e_2} t \overline{e_2}^{-1},\; \cdots\]
\end{proof}

As $G_c(m,n)$ is nilpotent of class $\leq c$, $f$ induces a surjective morphism 
\[BS_c(m,n)= \frac{BS(m,n)}{\gamma_{c+1}(BS(m,n))} \longrightarrow G_c(m,n).\]

For $G_c(m,n)$ we have that $\tau G= A$ and so $[ \tau G ,\tau G] =1$.
\begin{proposition}\label{isomorphism}
The morphism $f:BS(m,n) \rightarrow G_c(m,n)$ induces an isomorphism
\[ \mu : \frac{BS_c(m,n)}{[\tau BS_c(m,n),\tau BS_c(m,n)]}\rightarrow G_c(m,n).\]
\end{proposition}
\begin{proof}
We already explained that $f$ induces a morphism $\nu: BS_c(m,n)\rightarrow G_c(m,n)$. Of course 
$\nu(\tau BS_c(m,n)) \subseteq \tau G$ and so $\nu[\tau BS_c(m,n),\tau BS_c(m,n)]\subseteq [\tau G , \tau G]=1$. Therefore, there is an induced morphism 
\[ \mu:  \frac{BS_c(m,n)}{[\tau BS_c(m,n),\tau BS_c(m,n)]}\rightarrow G_c(m,n).\]
As $f$ is surjective, we know that $\mu$ is surjective too. 
In Lemma~\ref{hirsch}, we showed that $BS_c(m,n)$ has Hirsch length 1. Then also the quotient $BS_c(m,n)/[\tau BS_c(m,n),\tau BS_c(m,n)]$
has Hirsch length 1, since we 
take the quotient by a finite subgroup.  As also, by construction, $G_c(m,n)$ has Hirsch length 1 and $\mu$ is surjective, we must have that the Kernel of $\mu$ has Hirsch length 0, i.e.\ the Kernel of $\mu$ has to be finite.
For sake of simplicity we introduce the following notation:
\[ H=    \frac{BS_c(m,n)}{[\tau BS_c(m,n),\tau BS_c(m,n)]} \]
We already know, by Lemma~\ref{hirsch}, that 
$\tau H$ is generated by $\bar b$ and $\gamma_2(H)$. (Now, $\bar b$ denotes the image of $b$ in $H$).
As $\mu$ is surjective and has finite kernel (so Ker$(\mu)\subseteq \tau H$), we know that $\mu(\tau H)= \tau G_c(m,n)$. Therefore,
in order to prove that $\mu$ is injective, it is enough to show that $\# \tau H \leq \#\tau G_c(m,n)=|m-n|^c$. 

To be able to find a bound on $\# \tau H$, we look at the quotiens $\gamma_i(H)/\gamma_{i+1}(H)$.
\begin{itemize}
\item $\gamma_2(H)/\gamma_3(H)$ is generated by $[\bar a,\bar b]\gamma_3(H)$.
\item Then, $\gamma_3(H)/\gamma_4(H)$ is generated by 
$[\bar a ,[ \bar a, \bar b]] \gamma_4(H)$ and $[\bar b , [\bar a, \bar b]]\gamma_4 (H)$. However, in $H$ we 
have that $[\bar b , [\bar a, \bar b]]=1$ (since we divide out $[\tau BS_c(m,n), \tau BS_c(m,n)]$).\\
So $\gamma_3(H)/\gamma_4(H)$ is generated by $[\bar a ,[ \bar a, \bar b]] \gamma_4(H)$.
\item Continuing by induction, we find that 
$\gamma_i(H)/\gamma_{i+1}(H)$ is a cyclic group generated by 
\[  [\bar a,[ \bar a, [\bar a, \ldots ,[\bar a ,\bar b]]]] \gamma_{i+1}(H) \mbox{  (with $i-1$ times $\bar a$)}.\]
\end{itemize}
We already know that $\# \tau H /\gamma_2(H)=|m-n|$ (so $\bar b^{m-n}\gamma_2(H)=1\gamma_2(H)$, see the proof of Lemma~\ref{hirsch}).
Let $c_1=\bar{b}$ and for $i>1$ we let 
$c_i=   [\bar a,[ \bar a, [\bar a, \ldots ,[\bar a ,\bar b]]]]$ (with $i-1$ times $\bar{a}$).
Then $\gamma_i(H)/\gamma_{i+1}(H)$ is generated by $c_i\gamma_{i+1}(H)$ for $i>1$ and 
$\tau(H)/\gamma_2(H)$ is generated by $c_1\gamma_2(H)$.\\
We now show by induction on $i$ that $c_i^{m-n} \gamma_{i+1}(H)= 1 \gamma_{i+1} (H)$ and hence
$\# \gamma_i(H)/\gamma_{i+1}(H)\leq |m-n|$. 
We already obtained the case $i=1$. Now assume, 
the result holds for $c_{i-1}$ (with $i>1$), then $c_i =  [\bar a, c_{i-1}]$ and we have that 
\[ c_i^{m-n}\gamma_{i+1}(H) = [\bar a , c_{i-1}]^{m-n} \gamma_{i+1}(H) = 
[\bar a , c_{i-1}^{m-n}] \gamma_{i+1}(H) = 1 \gamma_{i+1}(H).\] 
As a conclusion, we find that 
\[ \# \tau(H) = \# \frac{\tau H}{\gamma_2(H)} \times 
                           \# \frac{\gamma_2(H)}{\gamma_3(H)}
                           \times \cdots \times
                           \# \frac{\gamma_c(H)}{\gamma_{c+1}(H)} \leq |m-n|^c= \# \tau G_c(m,n).\]
We can conclude that $\mu$ is injective (and hence an isomorphism).
\end{proof}

\begin{corollary}\label{isomorph} In case $\gcd(m,n)=1$ the morphism $f:BS(m,n) \rightarrow G_c(m,n)$ induces an isomorphism
\[BS_c(m,n)\cong G_c(m,n).\]
\end{corollary}
\begin{proof}
It follows from Corollary~\ref{metabelian} that in this case $[\tau BS_c(m,n),\tau BS_c(m,n)]=1$.
\end{proof}
\section{The case where $\gcd(m,n)=1$.}

In the next  lemma, we will make use of the Smith normal form, details about this normal form can e.g.\ be found in  \cite{S}.   
\begin{lemma}
Let $a,b\in \Z$ with $\gcd(a,b)=1$. Then, the Smith normal form of the $n\times n$--matrix
\[A_n=\left( \begin{array}{ccccccc}
a & 0 & 0 & 0 & \cdots &  0 & 0 \\
b & a & 0 & 0 & \cdots & 0 & 0 \\
0 & b & a & 0 & \cdots & 0 & 0 \\
0 & 0 & b & a & \cdots & 0 & 0 \\
\vdots & \vdots & \vdots & \vdots & \ddots & \vdots & \vdots \\
0 & 0 & 0 & 0 &\cdots & a & 0 \\
0 & 0 & 0 & 0 & \cdots & b & a \end{array}\right)\]
equals
\[ \left( \begin{array}{cccccc}
1 & 0 & 0 &  \cdots &  0 & 0 \\
0 & 1 & 0 &  \cdots & 0 & 0 \\
0 & 0 & 1 &  \cdots & 0 & 0 \\
 \vdots & \vdots & \vdots & \ddots & \vdots & \vdots \\
0 & 0 & 0 &\cdots & 1 & 0 \\
0 & 0 & 0 & \cdots & 0 & a^n \end{array}\right)\]
\end{lemma}
\begin{proof}
We will prove a slightly more general version of this lemma  and consider for any 
positive integer $k$ the matrix $A_n(k)$ which is the same matrix as $A_n$, except that the first entry of $A_n(k)$ (so on the first row and the first column) is $a^k$ 
  instead of $a$. So $A_n=A_n(1)$. 
We will now show by induction on $n$, that the Smith normal form of $A_n(k)$ equals
\[ \left( \begin{array}{ccccc}
1 & 0 &   \cdots &  0 & 0 \\
0 & 1 &   \cdots & 0 & 0 \\
 \vdots & \vdots &  \ddots & \vdots & \vdots \\
0 & 0 &\cdots & 1 & 0 \\
0 & 0 & \cdots & 0 & a^{n-1+k} \end{array}\right)\]
When $n=1$, there is nothing to show, so we assume that $n>1$.\\
As $\gcd(a,b)=1$, there exist integers $\alpha,\beta \in \Z$ such that $\alpha a^k + \beta b=1$. 
Now, consider 
\[ P=\left( \begin{array}{cc} \alpha & \beta \\ - b & a^k\end{array}\right) \in \GL(2,\Z) \mbox{ and }
Q=\left( \begin{array}{cc} 1 & - a \beta \\0 & 1 \end{array}\right) \in \GL(2,\Z).\]
It is now easy to compute that (with $I_{n-2}$ the $(n-2)\times (n-2)$ identity matrix)
\[ \left( \begin{array}{cc}
P & 0 \\ 0 & I_{n-2} \end{array}\right) A_n(k) \left( \begin{array}{cc}
Q & 0 \\ 0 & I_{n-2} \end{array}\right)=
\left( \begin{array}{cc} 1 & 0 \\ 0 & A_{n-1}(k+1) \end{array}\right). \]  
By induction, we know the Smith normal form of $A_{n-1}(k+1)$ and hence also of 
$ \left( \begin{array}{cc} 1 & 0 \\ 0 & A_{n-1}(k+1) \end{array}\right)$, which is then exactly 
\[ \left( \begin{array}{ccccc}
1 & 0 &   \cdots &  0 & 0 \\
0 & 1 &   \cdots & 0 & 0 \\
 \vdots & \vdots &  \ddots & \vdots & \vdots \\
0 & 0 &\cdots & 1 & 0 \\
0 & 0 & \cdots & 0 & a^{n-1+k} \end{array}\right)\]
as claimed.
\end{proof}

\begin{corollary}
Let $m,n$ be two integers with $\gcd(m,n)=1$ and $m\neq n$. Then 
\[ A_c(m,n)\cong \Z_{|m-n|^c}.\]
\end{corollary}
\begin{proof}
Recall that $A_c(m,n)=\displaystyle \frac{\Z^c}{{\rm Im}\, \varphi}$ where $\varphi:\Z^c\to \Z^c$ is represented by the matrix $\eqref{matrix-phi}$. The lemma above shows that the Smith normal form of this matrix is
\[ \left( \begin{array}{ccccc}
1 & 0 &   \cdots &  0 & 0 \\
0 & 1 &   \cdots & 0 & 0 \\
 \vdots & \vdots &  \ddots & \vdots & \vdots \\
0 & 0 &\cdots & 1 & 0 \\
0 & 0 & \cdots & 0 & |n-m|^c \end{array}\right)\]
from which the result follows.
\end{proof}

For the rest of this section we will indeed assume that $m\neq n$ and that $\gcd(m,n)=1$ (so $BS_c(m,n)\cong G_c(m,n)$, see Corollary~\ref{isomorph}), moreover
we will use $s=\overline{e_1}$ to denote the canonical projection of the first standard generator 
of $\Z^c$ in the group $A_c(m,n)$. As $A_c(m,n)$ is a cyclic group and $A_c(m,n)$ is generated as a $\langle t \rangle$--module by $s$, it follows that $s$ is also a generator of $A_c(m,n)$ as a cyclic group.
So
\[  \langle s \rangle=A_c(m,n)\cong \Z_{|n-m|^c} \mbox{ \ and \ } G_c(m,n) = \langle s \rangle \rtimes \langle t \rangle.\]
\begin{proposition}\label{nu}
Let $m\neq n $ and $\gcd(m,n)=1$, then $ G_c(m,n) = \langle s \rangle \rtimes \langle t \rangle$, with
\[ t^{-1}s t = s^\nu,\]
where $\nu\in\Z$ is an integer satisfying 
\begin{itemize}
\item $\gcd(\nu,n-m)=1$ and
\item $\nu m \equiv n \bmod |n-m|^c$.
\end{itemize}
\end{proposition}
\begin{proof}
We already explained that $G_c(m,n)=\langle s \rangle \rtimes \langle t \rangle$.  As $s$ is a generator of the cyclic group of order 
$|n-m|^c$, we must have that also $t^{-1} s t$ is a generator of $\langle s \rangle$, which implies that 
$t^{-1} s t = s^\nu$ for some integer $\nu$ with $\gcd(\nu, n-m)=1$. 

As we already saw in the proof of Lemma~\ref{relation} (recall $s=\overline{e_1}$), we also have that 
\[ t^{-1} s^m t = s^n.\]
As $t^{-1} s^m t = s^{\nu m}$, it follows that $s^{\nu m}= s^n$, hence
\[ \nu m \equiv n \bmod |n-m|^c.\]
\end{proof}

Let $\varphi:G_c(m,n) \to G_c(m,n)$ be an automorphism, then, since $\langle s \rangle=\tau G_c(m,n)$, $\varphi$ induces an automorphism 
\[ \bar\varphi: G_c(m,n)/\langle s \rangle= \langle t \rangle \cong \Z \to  G_c(m,n)/\langle s \rangle= \langle t \rangle \cong \Z.\]
So $\bar\varphi(t) = t^{\pm 1}$. The following lemma is easy to check 
\begin{lemma}\label{minus1}
With the notation above, we have that 
\[ R(\varphi)<\infty \Leftrightarrow \bar\varphi(t)=t^{-1}.\]
\end{lemma}
It follows that $G_c(m,n)$ does not have the $R_\infty$--property if and only if there exists an automorphism 
$\varphi$ of $G_c(m,n)$ such that $\bar\varphi(t)=t^{-1}$.
We are now ready to prove the main theorem of this section which gives us the $R_\infty$--nilpotency degree of any 
Baumslag--Solitar group which is determined by coprime parameters $m$ and $n$. 
\begin{theorem} \label{main1}
Let $m,n$ be integers with $0<m\leq |n|$ and $\gcd(m,n)=1$. Let $p$ denote the largest integer such that 
$2^p|2 m +2$. Then, the $R_\infty$--nilpotency degree $r$ of $BS(m,n)$ is given by 
\begin{itemize}
\item In case $n<-1$ then $r=2$.
\item In case $n=-1$ (so $m=1$) then $r=\infty$.
\item In case $n=m$ (so $n=m=1$) then $r=\infty$.
\item In case $n-m=1$, then $r=\infty$.
\item In case $n-m=2$, then $r=p+2$.
\item In case $n-m\geq 3$, then $r=2$.

\end{itemize}
\end{theorem}

\begin{proof} Let $m$ and $n$ be as in the statement of the theorem. \\
{\bf Let  $m=n$.} Then the fact that $\gcd(m,n)=1$ and $m>0$, implies that $m=n=1$. We have that $BS(1,1)=\Z^2=BS_c(m,n)$ (for all $c$) does not have the $R_\infty$--property, from 
which it follows that in this case the $R_\infty$--nilpotency index is $\infty$.

\medskip

So from now onwards we assume that $m\neq n$. We have to examine for which $c$, the group 
$G_c(m,n)$ has the $R_\infty$--property. So, we have to investigate, when $G_c(m,n)$ admits an automorphism $\varphi$ 
with $\bar\varphi(t)=t^{-1}$ (Lemma~\ref{minus1}). Such a morphism $\varphi$ satisfies
\begin{equation}\label{muandbeta}
\varphi(s)=s^\mu \mbox{ and } \varphi(t)=s^\beta t^{-1}\mbox{ for some }\mu, \beta\in \Z.
\end{equation}
In fact, given $\mu,\beta\in \Z$, the expressions  of \eqref{muandbeta} above determine an endomorphism 
of $G_c(m,n)$ if and only if the relation $t^{-1}st=s^\nu$ (where $\nu$ is as in Proposition~\ref{nu}) is preserved, i.e.\ it must hold that
\begin{eqnarray}
\varphi(t)^{-1} \varphi(s) \varphi(t) & = & \varphi(s)^{\nu}  \nonumber\\
& \Updownarrow & \nonumber\\
ts^{-\beta} s^\mu s^\beta t^{-1} &  = & s^{\mu\nu}\nonumber \\
& \Updownarrow & \nonumber \\
s^{\mu} &= & t^{-1}s^{\mu \nu}t = s^{\mu \nu^2} 
\end{eqnarray}
Moreover, such a $\varphi$ is an automorphsim if $s^\mu$ is a generator of $\langle s \rangle $, i.e.\ 
when $\gcd(\mu, |n-m| )=1$. In this case, the last condition is equivalent to 
\[  \nu^2 \equiv 1 \bmod |n-m|^c.\]
Moreover, as we also have that $\gcd(m,|n-m|)=1$, this is also equivalent to the requirement that 
\[ \nu^2 m^2 \equiv m^2 \bmod |n-m|^c.\]
Finally, using Proposition~\ref{nu}, which says that $\nu m \equiv n \bmod |n-m|^c$, we find that

\begin{center}
$G_c(m,n)$ does not have the $R_\infty$--property \\
$\Updownarrow$\\
$n^2\equiv m^2 \bmod |n-m|^c$\\
$\Updownarrow$\\
$n+m \equiv 0 \bmod |n-m|^{c-1}$
\end{center} 
So from now on we have to examine when the condition
\begin{equation}\label{theequation}
n+m\equiv 0 \bmod |n-m|^{c-1}
\end{equation}
is satisfied. 

\medskip

When $c=1$, the equation is always satisfied (reflecting the fact that finitely generated abelian groups do not have the 
$R_\infty$--property). 
So from now onwards we consider the case $c>1$.

\medskip

{\bf Let  $n=-m$.}  In this case   $n=-1$ and $m=1$, since $\gcd(m,n)=1$, then equation \eqref{theequation} is always 
satisfied. This shows that $BS(-1,1)$ (which is the fundamental group of the Klein Bottle) has an infinite $R_\infty$--nilpotency degree (although $BS(-1,1)$ does have the $R_\infty$--property \cite[Theorem 2.2]{GW1}). 

\medskip

{\bf Let  $n<-1$.} Then $|n-m|^{c-1}=(|n|+m)^{c-1} > |n+m|\neq 0$. This implies that the equation \eqref{theequation} is never satisfied. This means that in this case, the $R_\infty$--nilpotency degree of $BS(n,m)$ is 2.

\medskip

Now, we consider the case of positive $n$, where we already treated the case when $n=m$. So we have that $n=m+k$ for $k>0$.
Moreover, as $\gcd(n,m)=1$, we also have that $\gcd(k,m)=1.$
If equation \eqref{theequation} is satisfied, then $|n-m|=k$ divides $n+m=2 m+ k$, so $k|2m$ and as $\gcd (k,m)=1$, we must 
have that $k|2$, so $k=1$ or $k=2$.  

{\bf Let $n=m+k$ for $k\geq 3$.} From the considerations of the pargraph above we have that the $R_\infty$--nilpotency degree of $BS(m+k,k)$ is 2.

\medskip

{Let $n=m+1$.} In this case the equation \eqref{theequation} is again satisfied for all $c$ and hence the $R_\infty$--nilpotency degree of $BS(m+1,m)$ is $\infty$.

\medskip

{\bf Finally, let $n=m+2$.} Then equation \eqref{theequation} is of the form $2 m+2 \equiv 0 \bmod 2^{c-1}$.   This equation is satisfied exactly when $c\leq p+1$. It follows that the $R_\infty$--degree of 
$BS(m+2,m)$ (where $m$ is odd) is $p+2$. This finishes the proof.
\end{proof}
\section{The case where $\gcd(m,n)\neq 1$.}

\begin{lemma} Let $m,n$ be non-zero integers with $m\neq n$. If $d=\gcd(m,n)$, then 
\[ A_c(m,n)\cong \Z_d^c \oplus \Z_{\left|\frac{n-m}{d}\right|^c}.\]
\end{lemma}
\begin{proof}
Note that the matrix $\eqref{matrix-phi}$ equals 
\[ d \left( \begin{array}{cccccc}
\frac{n-m}{d} & 0 & 0 & \cdots & 0 & 0 \\
-\frac{m}{d} &\frac{ n-m}{d} & 0 & \cdots & 0 & 0 \\
0 & -\frac{m}{d} & \frac{n-m}{d} & \cdots & 0 & 0\\
\vdots & \vdots & \vdots & \ddots & \vdots & \vdots  \\
0 & 0 & 0 & \cdots &\frac{n-m}{m}& 0\\
0 & 0 & 0 & \cdots & -\frac{m}{d} & \frac{n-m}{d}
\end{array}\right).
\]
with $\gcd(\frac{m}{d}, \frac{n-m}{d})=1$. It follows that the Smith normal form of $\eqref{matrix-phi}$ equals
\[ \left( \begin{array}{ccccc}
d & 0 &   \cdots &  0 & 0 \\
0 & d &   \cdots & 0 & 0 \\
 \vdots & \vdots &  \ddots & \vdots & \vdots \\
0 & 0 &\cdots & d & 0 \\
0 & 0 & \cdots & 0 & d \left|\frac{n-m}{d}\right|^c \end{array}\right),\]
from which the result follows.
\end{proof}
It follows that $d A_c(m,n)\cong \Z_{\left|\frac{n-m}{d}\right|^c}$ is a cyclic subgroup of $A_c(m,n)$ and as this subgroup
is invariant under the action of $\langle t \rangle$, the semidirect product $(d A_c(m,n))\rtimes \langle t \rangle$, is a subgroup
of $G_c(m,n)$.  
\begin{lemma}
Let $0<m \leq |n|$ with $m\neq n$ and take $d=\gcd(m,n)$. 
Then $(d A_c(m,n) )\rtimes \langle t \rangle$ is a subgroup of $G_c(m,n)$ and 
we have that 
\[ (d A_c(m,n) )\rtimes \langle t \rangle \cong G_c\left( \frac{m}{d}, \frac{n}{d}\right).\]
\end{lemma}
\begin{proof}
The fact that $(d A_c(m,n) )\rtimes \langle t \rangle$ is a subgroup of $G_c(m,n)$ was already discussed before the statement of the lemma. Let $\varphi':\Z^c \to \Z^c$ be the morphism represented by the matrix, 
\[ \left( \begin{array}{cccccc}
\frac{n-m}{d} & 0 & 0 & \cdots & 0 & 0 \\
-\frac{m}{d} &\frac{ n-m}{d} & 0 & \cdots & 0 & 0 \\
0 & -\frac{m}{d} & \frac{n-m}{d} & \cdots & 0 & 0\\
\vdots & \vdots & \vdots & \ddots & \vdots & \vdots  \\
0 & 0 & 0 & \cdots &\frac{n-m}{m}& 0\\
0 & 0 & 0 & \cdots & -\frac{m}{d} & \frac{n-m}{d}
\end{array}\right),
\]
then $\varphi= d \varphi'$. We have that 
\[ d A_c(m,n)= d \frac{\Z^c}{{\rm Im}\,\varphi} = \frac{(d\Z)^c}{d {\rm Im }\,\varphi'}\cong \frac{\Z^c}{ {\rm Im }\,\varphi'}=
A_c\left( \frac{m}{d}, \frac{n}{d}\right).\]
It is now easy to see that under the identification $d A_c(m,n) \cong  A_c\left( \frac{m}{d}, \frac{n}{d}\right)$ the  action of
$t$ is still the same as what we had before and so 
\[ (d A_c(m,n) )\rtimes \langle t \rangle \cong G_c\left( \frac{m}{d}, \frac{n}{d}\right).\]

\end{proof}
\begin{lemma} Let $0<m \leq |n|$ with $m\neq n$ and take $d=\gcd(m,n)$.\\
Then, if $G_c\left( \frac{m}{d}, \frac{n}{d}\right)$ has property $R_\infty$, then also $G_c(m,n)$ has property $R_\infty$.
\end{lemma}

\begin{proof}
First let us remark that for any $\alpha\in A_c(m,n)$ there is an automorphism $\psi_\alpha$ of $G_c(m,n)=A_c(m,n)\rtimes \langle t\rangle$ with $\psi_\alpha(a)=a,\;\forall a \in A_c(m,n)$ and $\psi_\alpha(t)= t \alpha $.

Now suppose that $\psi\in \Aut(G_c(m,n))$ is an automorphism with $R(\psi)<\infty$. This means that 
$\psi(t)=\alpha t^{-1}$ for some $\alpha\in A_c(m,n)$. After composing $\psi$ with $\psi_\alpha$, we may assume that 
$\psi(t)=t^{-1}$. Since we also have that $\psi ( d A_c(m,n))=d A_c(m,n)$ (since $A_c(m,n)$ is a characteristic 
subgroup of $G_c(m,n)$), we have that $\psi$ restricts to an automorpism of $(d A_c(m,n) )\rtimes \langle t \rangle \cong G_c\left( \frac{m}{d}, \frac{n}{d}\right)$ with finite Reidemeister number. 
This shows that if $G_c(m,n)$ does not have property $R_\infty$, then also 
$G_c\left( \frac{m}{d}, \frac{n}{d}\right)$ does not have this property.
\end{proof}

The main result of this section is the following result.

\begin{theorem}
Let $0<m \leq |n|$  and take $d=\gcd(m,n)$. 
Let $p$ denote the largest integer such that 
$2^p|2 \frac{m}{d} +2$. Then, the $R_\infty$--nilpotency degree $r$ of $BS(m,n)$ is given by 
\begin{itemize}
\item In case $n<0$ and $n\neq -m$, then $r=2$.
\item In case $n=-m$ then $r=\infty$.
\item In case $n=m$  then $r=\infty$.
\item In case $n-m=d$, then $r=\infty$.
\item In case $n-m=2d$, then $2\leq r\leq p+2$.
\item In case $n-m\geq 3d$, then $r=2$.
\end{itemize}
\end{theorem}
\begin{remark}
As $d=\gcd(m,n)$, the difference $n-m$ is a multiple of $d$, so the theorem above does treat all possible cases.
\end{remark}
\begin{proof}
We will first deal with a few special cases and then treat the general case.

\medskip

{\bf  Let $n=m$.} In this case, there is an automporphism $\psi$ of $BS(m,m)$ mapping $a$ to $a^{-1}$ and $b$ to $b^{-1}$.
This automorphism induces minus the  identity map on $BS_1(m,m)=\frac{\gamma_1(BS(m,m))}{\gamma_2(BS(m,m))}\cong \Z^2$.  It now follows that the induced map $\bar\psi$ on any quotient $BS_c(m,m)$ has $-1$ as an eigenvalue and 
hence $R(\bar\psi)<\infty$ ( See \cite{DG2}). It follows that the $R_\infty$--nilpotency index of $BS(m,m)$ is $\infty$.

\medskip

{\bf Let  $n=-m$.}  Now consider the automorphism $\psi$ of $BS(m,-m)$ mapping $a$ to $a^{-1}$ and $b$ to $b$. Then $\psi$ induces a map on $BS_1(m,-m)\cong \Z \oplus \Z_{2m}$, which is minus the identity on the $\Z$--factor and the  identity on the 
$\Z_{2m}$ factor. The same argument as in the previous case now allows us to conclude that the $R_\infty$--nilpotency index of $BS(m,-m)$ is $\infty$.

\medskip

{\bf  Let  $n=m+d$.} So $m= kd $ and $n=(k+1)d$ for some positive integers $k$ and $d$. \\
We claim that $b^d\in \gamma_c(BS(kd,(k+1)d))$ for all $c\geq 1$. This claim is certainly correct for $c=1$.\\
Now, fix $c\geq 1$ and assume that  that $b^d\in \gamma_c(BS(kd, (k+1) d))$. Then also 
$b^{kd}\in  \gamma_c(BS(kd, (k+1) d))$ and hence 
\[ [a, b^{kd} ] \in \gamma_{c+1}(BS(kd, (k+1) d)).\]
But as $a^{-1} b^{kd} a= b^{(k+1) d}$, we have that $[a,b^{kd}]= a^{-1} b^{-kd} a b^{kd}= b^{- d}\in \gamma_{c+1}(BS(kd, (k+1) d))$.
By induction, this finishes the proof of the claim. 

It follows that  the group $BS(kd, (k+1)d)$ and the goup 
\[ C(d)=\langle a,b\, |\, a^{-1} b^{kd} a = b^{(k+1)d},\; b^d\rangle =  \langle a,b\, | \, b^d\rangle\]
have isomorphic nilpotent quotients, i.e.
\[ \frac{BS(kd, (k+1)d)}{\gamma_{c+1}(BS(kd, (k+1)d)  )}\cong 
\frac{C(d)}{\gamma_{c+1} (C(d))}.\]
Now, it is easy to see that $C(d)$ has an automorphism $\psi$ mapping $a$ to $a^{-1}$ and $b$ to $b$, such that 
$\psi$ induces an automorphism $\bar\psi$ on $\frac{C(d)}{\gamma_{c+1} (C(d))}$ with finite Reidemeister number.
It follows that the $R_\infty$--nilpotency degree of $BS(kd, (k+1)d)$ is infinite.

\medskip

{\bf All the other cases.} As a finitely generated abelian group never has the $R_\infty$--property, we have that 
$r\geq 2$. Now, assume that $\psi$ is an automorphims of $BS_c(m,n)$ with $R(\psi)<\infty$. Then, $\psi$ induces an automorphism $\bar\psi$ of $G_c(m,n)$ (since we divide out a characteristic subgroup to go from $BS_c(m,n)$ to $G_c(m,n)$ by 
Proposition~\ref{isomorphism}) with $R(\psi)<\infty$. It follows that the $R_\infty$--nilpotentcy degree is bounded above by the 
smallest $c$ for which $G_c(m,n)$ has property $R_\infty$. In turn, this number is bounded above by the smallest $c$ 
such that $G_c\left( \frac{m}{d}, \frac{n}{d}\right)$ has the $R_\infty$--property. This is exactly what we determined in the 
proof of Theorem~\ref{main1}, which finishes the proof. 
\end{proof}


\begin{thebibliography}{99}



\bibitem{DG} Dekimpe, K.;  Gon\c{c}alves, D.: The $R_\infty$ property for free groups, free nilpotent groups and
 free solvable groups. Bull. Lond. Math. Soc.  46 no. 4, 737-746  (2014).


\bibitem{DG1} Dekimpe, K.; Gon\c calves, D.: The  $R_{\infty}$ property   for abelian groups.  TMNA  46 no. 2,    
 773-784  (2015).  



\bibitem{DG2} Dekimpe, K.; Gon\c calves, D.:  The $R_{\infty}$  property for nilpotent quotients of surface groups. 
 Trans. London Math. Soc.  3 (1),  28-46  (2016).  doi: 10.1112/tlms/tlw002 



\bibitem{Fe}  Fel'shtyn, A. L.: The Reidemeister number of any automorphism of a Gromov hyperbolic group is infinite.
Zapiski Nauchnych Seminarov POMI 279, 229-241  (2001).




\bibitem{FG}  Fel'shtyn, A.L.; Gon\c calves, D. L.: The Reidemeister number of any automorphism of a Baumslag-Solitar group is infinite. Geometry and dynamics of groups and spaces, 399--414, Progr. Math., 265, Birkh\"auser, Basel, 2008.

\bibitem{FN} Fel'shtyn, A.L.;  Nasybullov, T.: The $R_\infty$ and  $S_\infty$ properties for linear algebraic groups.
 J. Group Theory  19  (2016),  no. 5, 901--921.

\bibitem{FT}  Fel'shtyn, A.L.;   Troitsky E.: Twisted conjugacy classes in residually finite groups
  arXiv:math/1204.3175v2 2012. 20 pp.



\bibitem{GW} Gon\c calves D.; Wong, P.:  Twisted conjugacy classes in wreath products. Internat. J. Algebra Comput. 16  no. 5, 875-886 (2006).


\bibitem{GW1} Gon\c calves D.; Wong, P.: Twisted conjugacy classes in nilpotent groups. J. Reine Angew. Math. 633, 11-27
(2009).

 
 \bibitem{JLL} Jo, J.H. ;  Lee, J.B. ;  Lee, S.R.: The $R_\infty$ property for Houghton's groups.
 Algebra Discrete Math.  23  (2017),  no. 2, 249--262.
 
\bibitem {Le}  Levitt, G.: On the automorphism group of generalized Baumslag--Solitar groups. Geom. Topol. 11, 
 473--515 (2007).
 
 
\bibitem {LL} Levitt, G.;  Lustig, M.: Most automorphisms of a hyperbolic group have very simple dynamics. Ann.
Scient. Ec. Norm. Sup. 33, 507-517 (2000).


\bibitem{S} Sims, C.:  Computation with finitely presented groups.
Encyclopedia of Mathematics and its Applications, 48. Cambridge University Press, Cambridge,  1994.


\bibitem{STW} Stein, M,;  Taback, J.;  Wong, P.:  Automorphisms of higher rank lamplighter groups.
 Internat. J. Algebra Comput.  25  (2015),  no. 8, 1275--1299
 
\end{thebibliography}
\end{document}